\documentclass[12pt]{amsart}
\usepackage{amssymb,latexsym}
\usepackage{enumerate}
\usepackage{fullpage}

\makeatletter \@namedef{subjclassname@2010}{%
  \textup{2010} Mathematics Subject Classification}
\makeatother

\newcounter{thm} \numberwithin{thm}{section}
\newtheorem{Theorem}[thm]{Theorem}
\newtheorem{Proposition}[thm]{Proposition}
\newtheorem{Lemma}[thm]{Lemma}

\newcommand{\CC}[0]{\mathbb C}

	\newcommand{\ZZ}[0]{\mathbb Z}

\renewcommand{\aa}[0]{\textbf{\textit{a}}}

\renewcommand{\tt}[0]{\textbf{\textit{t}}}

\renewcommand{\hat}[1]{\widehat{#1}}

\renewcommand{\mod}[1]{\ (\text{mod }#1)}

\newcommand{\lr}[1]{\left(#1\right)}

\newcommand{\Var}[0]{\text{Var}}

\begin{document}


\baselineskip=17pt



\title{Littlewood's problem for sets with multidimensional structure}

\author[Brandon Hanson]{Brandon Hanson} \address{University of Georgia\\
Athens, GA}
\email{brandon.w.hanson@gmail.com}
\date{}
\maketitle
\begin{abstract}
    We give $L^1$-norm estimates for exponential sums of a finite sets $A$ consisting of integers or lattice points. Under the assumption that $A$ possesses sufficient multidimensional structure, our estimates are stronger than those of McGehee-Pigno-Smith and Konyagin. These theorems improve upon past work of Petridis.
\end{abstract}
\section{Introduction}

Let $A$ be a finite set of integers. The relationship between the additive structure of $A$ and the exponential sum\footnote{Throughout, we will use the notation $e(z)=e^{2\pi i z}$.} \[F(t)=\sum_{a\in A}e(at)\] is well-documented (see \cite{TV}, Chapter 4). In particular, we have
\[\int_0^1\left|F(t)\right|^{2k}dt=|\{(a_1,\ldots,a_{2k})\in A^{2k}:a_1+\cdots+a_k=a_{k+1}+\cdots+a_{2k}\}|,\]
so that the larger the $2k$'th moment of $F$, the more additive structure $A$ possesses. For $1\leq p<2$, we expect the $L^p$-norms of $F$ to be smaller for additively structured sets. This led Littlewood to conjecture (see, for instance \cite{Littlewood}) that 
\begin{equation}\label{littlewoodConjecture}
\inf_{\substack{A\subseteq \ZZ\\ |A|=N}}\int_0^1\left|\sum_{a\in A}e(at)\right|dt=\int_0^1\left|\sum_{n=1}^Ne(nt)\right|dt.
\end{equation}
It was a great feat when the estimate
\begin{equation}\label{logBound}
    \inf_{\substack{A\subseteq \ZZ\\ |A|=N}}\int_0^1\left|\sum_{a\in A}e(at)\right|dt\geq C\int_0^1\left|\sum_{n=1}^Ne(nt)\right|dt
\end{equation}
was established by McGehee, Pigno and Smith in \cite{MPS} and independently by Konyagin in \cite{Konyagin} for some absolute constant $C>0$. Here we will record the theorem of McGehee, Pigno and Smith as it will be used repeatedly in this article.
\begin{Theorem}\label{MPS}
Let $a_1<\ldots<a_n$ be a sequence of integers and let $u_1,\ldots,u_n$ be complex numbers. Then
\[\int_0^1\left|\sum_{j=1}^nu_je(a_jt)\right|dt\geq C_{MPS}\sum_{j=1}^n\frac{|u_j|}{j},\] where $C_{MPS}>0$ is an absolute constant.
\end{Theorem}
The estimate (\ref{logBound}) leaves open a few questions. First, it remains to establish the sharp constant, i.e. to prove (\ref{littlewoodConjecture}). Second, given a positive constant $C$, one would like to characterize the sets $A$ for which \[
    \int_0^1\left|\sum_{a\in A}e(at)\right|dt\leq C\int_0^1\left|\sum_{n=1}^Ne(nt)\right|dt,
\] a question sometimes referred to as the \emph{Inverse Littlewood Problem}, see \cite{GreenICM}. This article concerns the latter problem and we interpret it as follows: if $A$ possesses some structure which is decidedly unlike an arithmetic progression, can the estimate (\ref{logBound}) be improved? Specifically, we will explore how the notion of dimension can be leveraged. Such questions have already been investigated by Petridis \cite{Petridis}, and where appropriate we will compare results.

The first notion of dimension we will explore is quite literal - we consider $A$ a subset of the lattice $\ZZ^r$. Since $\ZZ^r$ contains one-dimensional sets, one must take steps to ensure $A$ is truly multidimensional, which we now do. 

For $i=1,\ldots,r$, let $\pi_i:\ZZ^r\to \ZZ$ denote the $i$'th coordinate projection and let \[A_i=\pi_i(A)\] denote the image of $A$ under this projection;  for $a_i\in\pi_i(A)$ let \[A_i^*(a_i)=\pi_i^{-1}(a_i)\cap A\] denote the fibre of $A$ above $a_i$. Our first estimate extends Theorem \ref{MPS} to higher dimensional sets.

\begin{Theorem}\label{basicMultiDim}
Suppose $A\subseteq \ZZ^r$ and $A_1$ is ordered as \[A_1=\{a_{1,1}<\ldots<a_{1,n}\}.\] Then we have the estimate
\[\int_{[0,1]^r}\left|\sum_{\aa\in A}e(\aa\cdot\tt)\right|d\tt\geq C_{MPS}\sum_{j=1}^n\frac{1}{j}\int_{[0,1]^{r-1}}\left|\sum_{\aa^*\in A_1^*(a_{1,j})}e(\aa^*\cdot\tt)\right|d\tt.\]
\end{Theorem}

We say $A\subseteq \ZZ$ is $n$-strongly 1-dimensional if $|A|\geq n$. Inductively, if $(n_1,\ldots,n_r)$ is a $r$-tuple of natural numbers then we say a set $A\subseteq\ZZ^r$ is $(n_1,\ldots,n_r)$-strongly $r$-dimensional if $|A_1|\geq n_1$ and $|A_1^*(a_1)|$ is $(n_2,\ldots,n_r)$-strongly $(r-1)$-dimensional for each $a_1\in A_1$.

\begin{Theorem}\label{multiDim}
Suppose $A\subseteq\ZZ^r$ is a  $(n_1,\ldots,n_r)$-strongly $r$-dimensional subset of $\ZZ^r$. Then
\[\int_{[0,1]^r}\left|\sum_{\aa\in A}e(\aa\cdot\tt)\right|d\tt\geq C_{MPS}^r \log(n_1)\cdots\log(n_r).\]
\end{Theorem}

Here we have strived to make the dependence on the implicit constant from Theorem \ref{MPS} explicit. This estimate is an improvement on Theorem 1.2 in \cite{Petridis} and is best-possible up to the constant $C_{MPS}^r$.

We now move to the case of subsets $A$ of $\ZZ$, which are one dimensional but have a structure if higher dimensional sets. As motivation, recall that a \emph{generalized arithmetic progression of rank} $2$ is a set of the form
\[G=\{am+bn:1\leq m\leq M,\ 1\leq n\leq N\}.\] These sets arise as projections of boxes in $\ZZ^2$, hence we think of them as possessing multidimensional structure. To guarantee that the elements $am+bn$ are distinct, it is sufficient to impose the condition $aM<b$. It is this sort of condition, which can be viewed as a multiscale condition, that motivates the last theorem of this article. To state it, we begin with appropriate notions of a multidimensional subset of $\ZZ$. We say $A\subseteq\ZZ$ is $n$-strongly $1$-dimensional if $|A|\geq n$. For $\delta_1,\ldots,\delta_{r-1}>0$, we define inductively that a finite set $A\subseteq \ZZ$ is $(\delta_1,\ldots,\delta_{r-1};n_1,\ldots,n_r)$-strongly $r$-dimensional if 
there are numbers $d_1$ and $d_2$ with $d_2>(2+\delta_1)d_1$ and such that
\[A=\bigcup_{k\in I}A_k+kd_2\]
for some set $I$ of consisting of at least $n_1$ integers and subsets $A_k\subseteq\{-d_1,\ldots,d_1\}$ which are each $(\delta_2,\ldots,\delta_{r-1};n_2,\ldots,n_r)$-strongly $(r-1)$-dimensional.

To get a sense of this definition, it is best to think of strongly 2-dimensional sets; higher dimensional sets are handled by an iterative argument. The easiest example of a $(\delta,n)$-strongly 2-dimensional set is a union of intervals $A_k$ of length $n$, separated by gaps of length $\delta n$. One needs that $n$ is somewhat large - if $n$ were 1, then this set would be an arithmetic progression; one also requires the gaps between intervals - without them, our set would be an ordinary interval. However our definition allows, for instance, to pass to very sparse subsets of such a set. In general, each set $A_k$ can be compared to the fibres $A_1^*(a_k)$ in the $\ZZ^2$ setting, and then we only require that these fibres are sufficiently large (but independently of $d_1$). The condition that $d_2>(2+\delta)d_1$ imposes a gap of size at least $\delta d_1$ between any two $A_k$. This gap is a substitute for the fact that the fibres of a projection $\pi:\ZZ^2\to \ZZ$ are independent.

\begin{Theorem}\label{multiDimZ}
Let $\delta,\ldots,\delta_{r-1}>0$ and $n_1,\ldots,n_r$ be positive integers satisfying
\[n_i\geq \pi^32^{21}C_{MPS}^3\prod_{j=i}^r(\log(n_j))^3\]for each $i$.
Suppose $A$ is a $(\delta_1\ldots,\delta_{r-1};n_1,\ldots,n_r)$-strongly $r$-dimensional subset of $\ZZ$. Then
\[\int_0^1\left|\sum_{a\in A}e(at)\right|dt\geq C_{\delta_1\ldots,\delta_{r-1}}\log(n_1)\cdots\log(n_r),\]
where
\[C_{\delta_1\ldots,\delta_{r-1}}=C_{MPS}^r(2^9\pi)^{-r}\prod_{j=1}^{r-1}(2+\log(1+2/\delta_j))^{-1}.\]
\end{Theorem}

This theorem is also best-possible up to the  constant, as the lower bound is realized by an appropriately chosen $r$-dimensional arithmetic progression, see Theorem 3.3 in \cite{Shao}. Estimates for multidimensional subsets of $\ZZ$ were established in Theorem 1.3 of \cite{Petridis} as well. There, the bounds are likely not as sharp as in Theorem \ref{multiDimZ}, but the hypotheses are somewhat different, relying on the notion of a \emph{Freiman isomorphism}. It might also be noted that Theorem \ref{multiDimZ} holds for two-dimensional sets, which was not established in \cite{Petridis}.

\section*{Acknowledgements}

I am grateful for the help of Giorgis Petridis and \'Akos Magyar. They generously contributed ideas which improved the arguments in this article.

\section{Strongly multidimensional sets in $\ZZ^r$}
The basic argument relies on the solution to Littlewood's problem by McGehee-Pigno-Smith, and in particular their generalized Hardy inequality.

\begin{proof}[Proof of Theorem \ref{basicMultiDim}]
Let $(t_1,t_2,\ldots,t_r)\in[0,1]^{r}$ and write
\[\sum_{\aa\in A}e(\aa\cdot\tt)=\sum_{j=1}^ne(a_{1,j}t_1)\lr{\sum_{\aa^*\in A_1^*(a_1)}e(\aa^*\cdot(t_2,\ldots,t_r))}.\]
We interpret this as a trigonometric polynomial in the variable $t_1$ with complex coefficients. By Theorem \ref{MPS} we have
\[\int_0^1\left|\sum_{\aa\in A}e(\aa\cdot\tt)\right|dt_1\geq C_{MPS}\sum_{j=1}^n\frac{1}{j}\left|\sum_{\aa^*\in A_1^*(a_1)}e(\aa^*\cdot(t_2,\ldots,t_r))\right|.\] Integrating over $t_2,\ldots,t_r$ completes the proof.
\end{proof}

Iterated application of Theorem \ref{basicMultiDim} leads to the proof of Theorem \ref{multiDim}.

\begin{proof}[Proof of Theorem \ref{multiDim}]
We proceed by induction on $r$, and when $r=1$, this follows immediately from Theorem \ref{MPS}. By the preceding proposition, we have 
\[\int_{[0,1]^d}\left|\sum_{\aa\in A}e(\aa\cdot\tt)\right|d\tt\geq C_{MPS}\sum_{j=1}^n\frac{1}{j}\int_{[0,1]^{r-1}}\left|\sum_{\aa^*\in A_1^*(a_1)}e(\aa^*\cdot\tt)\right|d\tt,\] and each of the sets $A_1^*(a_1)$ is $(n_2,\ldots,n_r)$-strongly $(r-1)$-dimensional. By induction, 
\[\int_{[0,1]^{r-1}}\left|\sum_{\aa^*\in A_1^*(a_1)}e(\aa^*\cdot\tt)\right|d\tt\geq C_{MPS}^{r-1}\log(n_2)\cdots\log(n_r)\]
and since $n\geq r_1$ the theorem is proved.
\end{proof}

\section{Lemmata}
Throughout, we will say $f$ is a trigonometric polynomial of degree $d$ if
\[f(t)=\sum_{|n|\leq d}a_ne(nt).\] Notice $L^1$-norms are preserved if we translate the support of $\hat f$ by $d$.
We need the following.
\begin{Lemma}[Bernstein's inequality]\label{bernstein}
Let $f:[0,1]\to \CC$ be a trigonometic polynomial of degree $d$. Then
\[\|f'\|_{L^1([0,1])}\leq 2\pi d\|f\|_{L^1([0,1])}.\]
\end{Lemma}
\begin{proof}
See \cite[Chapter 1, Excercise 7.16]{Katznelson}.
\end{proof}

\begin{Lemma}\label{numerical}
Let $n$ be a positive integer. Then for any trigonometric polynomial $f$ of degree $d$ we have
\[\left|\|f\|_{L^1([0,1])}-\frac{1}{N}\sum_{j=1}^N|f(j/N)|\right|\leq \frac{4\pi d}{N}\|f\|_{L^1([0,1])}.\]
\end{Lemma}
\begin{proof}
We begin with the estimate,
\[\left|\int_0^1|f(t)|dt-\frac{1}{N}\sum_{j=1}^N|f(j/N)|\right|\leq \frac{1}{N}\Var(|f|)\]
where \[\Var(g)=\sup_{0=x_0<x_1<\ldots<x_M=1}\sum_{j=1}^M|g(x_j)-g(x_{j-1})|\]
is the total variation of $g$. This is a simple case of, for instance, Koksma's inequality, see \cite{KN}. Since
\[\Var(|f|)\leq \Var(\Re(f))+\Var(\Im(f))=\int_0^1|\Re(f)'(t)|dt+\int_0^1|\Im(f)'(t)|dt\leq 2\int_0^1|f'(t)|dt\]
the result now follows from Bernstein's inequality.
\end{proof}

One of the key ideas that goes into the proof of Theorem \ref{multiDimZ} is to amplify the gaps between the different pieces of $A$. To do so, we need a function that can isolate the various subsets $A_k$.

\begin{Lemma}\label{Kernel}
Let $M$ and $N$ be integers with $2\leq M<N$ and let $R\geq 2N+4M+1$. Then there is a function $K_{M,N}$ with the following properties:
\begin{enumerate}
    \item $K_{M,N}(k)=1$ for $|k|\leq N$,
    \item $K_{M,N}(k)=0$ for $|k|\geq N+2M$, and
    \item \[\frac{1}{R}\sum_{j=1}^{R}|\hat{K_{M,N}}(j/R)|\leq 32\pi(2+\log(1+N/M)).\]
\end{enumerate}
\end{Lemma}

\begin{proof}

Recall that the Dirichlet kernel of order $N$ is
\[D_N(t)=\sum_{|n|\leq N}e(nt)=\frac{\sin(\pi(2N+1)t)}{\sin(\pi t)},\]
and the Fejer kernel or order $N$ is
\[F_N(t)=\sum_{|n|\leq N}\lr{1-\frac{|n|}{N+1}}e(nt)=\frac{1}{N+1}\frac{(\sin(\pi (N+1)t))^2}{(\sin(\pi t))^2}.\]
We have
\[|D_N(t)|\leq 2N+1,\ 0\leq F_N(t)\leq N+1,\ \int_0^1 F_N(t)dt=1.\]
Now let $M$ and $N$ be integers with $M<N$ and define
\[K_{M,N}(k)=\frac{1}{M}\sum_{\substack{|n|\leq M-1\\ |n-k|\leq N+M}}\lr{1-\frac{|n|}{M}}.\]
Then
\[\hat{K_{M,N}}(t)=\frac{1}{M}D_{N+M}(t)F_{M-1}(t).\]
First, observe that if $|k|\leq N$ then 
\[K_{M,N}(k)=\frac{1}{M}\sum_{|n|\leq M-1}\lr{1-\frac{|n|}{M}}=1.\]
Meanwhile if $|k|\geq N+2M$ then $K_{M,N}(k)=0$ since the defining sum is empty. Thus we have (1) and (2).

For (3), we have
\[\int_0^1|\hat{K_{M,N}}(t)|dt=2I_1+2I_2+2I_3\]
where
\[I_1=\frac{1}{M}\int_0^\frac{1}{N+M}|D_{M+N}(t)|F_{M-1}(t)dt\leq 3,\]
\[I_2=\frac{1}{M}\int_\frac{1}{N+M}^\frac{1}{M}|D_{M+N}(t)|F_{M-1}(t)dt,\]
and
\[I_3=\frac{1}{M}\int_\frac{1}{M}^\frac12|D_{M+N}(t)|F_{M-1}(t)dt.\]
Using $2t\leq |\sin(\pi t)|\leq \pi t$ for $|t|\leq \frac{1}{2}$, we have
\begin{align*}
I_2&\leq\frac{1}{M^2}\int_\frac{1}{N+M}^\frac{1}{M}\frac{|\sin(\pi Mt)|^2}{|\sin(\pi t)|^3}dt\\
&\leq\frac{\pi^2}{8}\int_\frac{1}{N+M}^\frac{1}{M}\frac{dt}{t}\\
&\leq 2\log(1+N/M).
\end{align*}
Finally,
\[I_3\leq\frac{1}{M^2}\int_\frac{1}{M}^\frac{1}{2}\frac{1}{t^3}dt\leq 1.\]
By Lemma \ref{numerical},
\[\frac{1}{R}\sum_{j=1}^{R}|\hat{K_{M,N}}(j/R)|\leq 4\pi\|\hat {K_{M,N}}\|_{L^1([0,1])}\leq 8\pi(4+2\log(1+N/M)).\]
\end{proof}

\begin{Lemma}\label{periodic}
Let $R$ a positive integer and $K:\ZZ\to \CC$ be a periodic function with period $R$. Then
\[\int_0^1\left|\sum_{m}a_mK(m)e(mt)\right|dt\leq\frac{1}{R}\sum_{j=1}^R|\hat K(j/R)|\int_0^1\left|\sum_{m}a_me(mt)\right|dt.
\]
\end{Lemma}
\begin{proof}
By orthogonality of characters modulo $R$
\[\sum_{m}a_mK(m)e(mt)=\frac{1}{R}\sum_{j=1}^R\hat K(j/R)\sum_{m}a_me(m(t+j/R)).\]
So by the triangle inequality,
\begin{align*}
    \int_0^1\left|\sum_{m}a_mK(m)e(mt)\right|dt&\leq \frac{1}{R}\sum_{j=1}^R|\hat K(j/R)|\int_0^1\left|\sum_{m}a_me(m(t+y/R))\right|dt.
\end{align*}
\end{proof}

Given a set $I\subseteq\ZZ$, a positive integer $q$, and an arbitrary integer $s$, we define \[I(q;s)=\{k\in I:k=s\mod q\}.\]
The following lemma is used to amplify the space between the sets $A_k$.

\begin{Lemma}\label{thinning}
Let $d_1,d_2$ and $q$ be positive integers with $(2+2\delta)d_1+4\leq d_2$ for some $\delta>0$ and $q\geq 4$. Suppose $I$ is a finite set of integers, and let \[F(t)=\sum_{k\in I}f_k(t)e(d_2kt)\] where each $f_k$ is a trigonometric polynomial of degree at most $d_1$. Then for any integer $s$, we have
\[\int_0^1\left|\sum_{k\in I(q;s)} f_{k}(t)e(d_2kt)\right|dt\leq 32\pi(2+ \log(1+2/\delta))\|F\|_{L^1([0,1])}.\]
\end{Lemma}

\begin{proof}
If necessary we may replace $I$ with $I-s$ while preserving the $L^1$-norm, and so there is no loss of generality in assuming $s=0$. By definition, we can write
\[F(t)=\sum_{m}a_me(mt)\] where the coefficients $a_m$ are supported on numbers of the form \begin{equation}\label{support}
   m=d_2k+l,\ |l|\leq d_1.  
\end{equation}

Let $M=\lceil \delta d_1/2\rceil$ and $N=d_1$, and let $K_{M,N}$ be the function from Lemma \ref{Kernel}. The support of $K_{M,N}$ is contained in the interval \[[-N-2M,N+2M]\subseteq [-d_2/2,d_2/2]\] and $K_{M,N}$ is identically $1$ on $\{-d_1,\ldots,d_1\}$. Extend $K_{M,N}$ periodically with period $qd_2$. Then by Lemma \ref{periodic}, \[\int_0^1\left|\sum_{m}a_mK_{M,N}(m)e(mt)\right|dt\leq 32\pi(2+\log(1+N/M))\|F\|_{L^1([0,1])}.\]

Now $a_mK_{M,N}(m)$ is only non-zero if $m=jqd_2+l'$ with $-d_2/2 \leq l'\leq d_2/2$. By (\ref{support}), we have \[l-l'=d_2(jq-k),\] and since $|l-l'|\leq d_1+d_2/2<d_2$, this can only happen if $k=jq$ and $l=l'$. So we are left with coefficients supported on integers of the form $jqd_2+l$ with $-d_1\leq l\leq d_1$. However, $K_{M,N}$ is identically $1$ on numbers of the form $jqd_2+l$ with $-d_1\leq l\leq d_1$. In summary,
\[\sum_{m}a_mK_{M,N}(m)e(mt)=\sum_{\substack{k\in I\\k=0\mod q}}e(kd_2t)\sum_{|l|\leq d_1}a_{kd_2+l}e(lt)=\sum_{k\in I(q;0)}f_k(t)e(kd_2t).\]
\end{proof}

Finally, in order to apply Lemma \ref{thinning} effectively, we need a good modulus $q$. Such a modulus is guaranteed by the following lemma.

\begin{Lemma}\label{goodModulus}
Let $I$ be a set integers with $|I|\geq 8$. Then there are positive integers $q$ and $s$ such that 
\[|I|^{1/3}/8\leq |I(q;s)|\leq q^{1/2}\]
\end{Lemma}
\begin{proof}
For each $j\geq 1$, choose any $s_j$ so that $|I(4^j;s_j)|$ is maximal. Then $|I(4^j;s_j)|\geq 4^{-j}|I|$ by the pigeonhole principle. We have \[|I(4;s_1)|\geq |I|/4\geq 2\] while for sufficiently large $j$ we have $|I(4^{j};s_{j})|=1\leq 2^{j}$. It follows that there is a minimal $j_0$ so that $|I(4^{j_0};s_{j_0})|\leq 2^{j_0}$. We let $q=4^{j_0}$, and $s=s_{j_0}$. Then
\[\frac{|I|}{q}\leq |I(q,s)|\leq q^{1/2},\]
so that in particular $|I|^{1/3}\leq q^{1/2}$. By minimality of $j_0$
\[q^{1/2}/2=2^{j_0-1}\leq |I(q/4;s_{j_0-1})|\leq 4I(q;s),\]
so that $|I(q;s)|\geq |I|^{1/3}/8$.
\end{proof}
\section{Multidimensional subsets of $\ZZ$}

The following proposition can be viewed as a sort of analog to Theorem \ref{basicMultiDim}.
\begin{Proposition}\label{MainProp}
Let $d_1,d_2$ positive integers with $(2+\delta)d_1<d_2$. Suppose $I$ is a finite set of integers, and let \[F(t)=\sum_{k\in I}f_k(t)e(d_2kt)\] where \[f_k(t)=\sum_{|n|\leq d_1}a_{n,k}e(nt).\] Let $q$ and $s$ an integers with $q>4\pi$ and suppose \[I(q;s)=\{k_1<\ldots<k_J\}.\]
Then we have have \[\|F\|_{L^1([0,1])}\geq \frac{1}{32\pi(2+\log(1+2/\delta))}\sum_{j=1}^J\|f_{k_j}\|_{L^1([0,1])}\lr{\frac{C_{MPS}}{2j}-\frac{2\pi d_1}{qd_2}}.\]
\end{Proposition}
\begin{proof}
By the Lemma \ref{thinning}, we have
\[\|F\|_{L^1([0,1])}\geq \frac{1}{32\pi(2+\log(1+2/\delta))}\int_0^1\left|\sum_{j=1}^J f_{{k_j}}(t)e({k_j}d_2t)\right|dt.\]
Next we write each ${k_j}\in I(q;s)$ as ${k_j}=b_{k_j}q+s$, so the above integral becomes
\begin{align*}
    \int_0^1\left|\sum_{j=1}^J f_{{k_j}}(t)e((b_{k_j}q+s)d_2t)\right|dt&=\int_0^1\left|\sum_{j=1}^J f_{{k_j}}(t)e(b_{k_j}qd_2t)\right|dt\\
    &=\frac{1}{qd_2}\int_{0}^{qd_2}\left|\sum_{j=1}^J f_{{k_j}}(u/qd_2)e(b_{k_j}u)\right|du.
\end{align*}
after the change of variables $u=qd_2t$. 
Next, by breaking the integral into intervals of unit length, we get
\[\frac{1}{qd_2}\sum_{m=1}^{qd_2}\int_{m-1}^m\left|\sum_{j=1}^J f_{{k_j}}(u/qd_2)e(b_{k_j}u)\right|du\geq T_1-T_2\]
where\[T_1=\frac{1}{qd_2}\sum_{m=1}^{qd_2}\int_{m-1}^{m}\left|\sum_{j=1}^Jf_{{k_j}}((m-1)/qd_2)e(b_{k_j}u)\right|du\] and \[T_2=\frac{1}{qd_2}\sum_{m=1}^{qd_2}\int_{m-1}^{m}\left|\sum_{j=1}^J(f_{{k_j}}(u/qd_2)-f_{{k_j}}((m-1)/qd_2))e(b_{k_j}u)\right|du.\]

By periodicity followed by Theorem \ref{MPS}, we have the estimate
\[T_1=\frac{1}{qd_2}\sum_{m=1}^{qd_2}\int_{0}^{1}\left|\sum_{j=1}^Jf_{{k_j}}((m-1)/qd_2)e(b_{k_j}u)\right|du\geq \frac{C_{MPS}}{qd_2}\sum_{m=1}^{qd_2}\sum_{j=1}^J\frac{|f_{k_j}((m-1)/qd_2)|}{j}.\] An application of Lemma \ref{numerical}, bearing in mind $qd_2>8\pi d_1$, yields
\[\frac{1}{qd_2}\sum_{m=1}^{qd_2}|f_{k_j}((m-1)/qd_2)|\geq\frac{1}{2}\|f_{k_j}\|_{L^1([0,1])}.\]

To complete the proof of the proposition it suffices to estimate $T_2$ appropriately. For $u$ in the interval $m-1\leq u\leq m$ we have
\begin{align*}
\left|\sum_{j=1}^J(f_{k_j}(u/qd_2)-f_{k_j}((m-1)/qd_2))e(b_{k_j}u)\right|&=
\left|\sum_{j=1}^J\int_{\frac{m-1}{qd_2}}^\frac{u}{qd_2}f_{k_j}'(v)e(b_{k_j}u)dv\right|
\\
&\leq \int_{\frac{m-1}{qd_2}}^\frac{m}{qd_2}\left|\sum_{j=1}^Jf_{k_j}'(v)e(b_{k_j}u)\right|dv
\end{align*}
by the triangle inequality and positivity. Thus we get the upper bound
\begin{align*}
    T_2&\leq \frac{1}{qd_2}\sum_{m=1}^{qd_2}\int_{\frac{m-1}{qd_2}}^\frac{m}{qd_2}\int_{m-1}^m\left|\sum_{j=1}^J e(b_{k_j}u)f_{k_j}'(v)\right|dudv\\
    &=\frac{1}{qd_2}\int_{0}^1\int_{0}^1\left|\sum_{j=1}^J e(b_{k_j}u)f_{k_j}'(v)\right|dudv\\
    &\leq\frac{1}{qd_2}\sum_{j=1}^J\int_{0}^1 \left|f_{k_j}'(v)\right|dv.\end{align*}
By Lemma \ref{bernstein}, we obtain
\[T_2\leq \frac{2\pi d_1}{qd_2}\sum_{j=1}^J\|f_{k_j}\|_{L^1([0,1])}.\]
\end{proof}

\begin{proof}[Proof of Theorem \ref{multiDimZ}]
As in the proof of \ref{multiDim}, we proceed by induction on $r$. When $r=1$, this follows immediately from Theorem \ref{MPS}.

For the inductive step, we begin by writing $A$ as
\[A=\bigcup_{k\in I}A_k+kd_2\]
for some $(n_2,\ldots,n_r)$-strongly regular sets $A_k\subseteq\{-d_1,\ldots,d_1\}$, where $d_1$ and $d_2$ with $16<4d_1<d_2$. If we let
\[f_k(t)=\sum_{a\in A_k}e(at)\]
then
\[\sum_{a\in A}e(at)=\sum_{k\in I}f_k(t)e(d_2kt)\]
is of the necessary form to apply Proposition \ref{MainProp}. If it is the case that  
\[\|f_k\|_{L^1([0,1])}\geq C_{\delta_2,\ldots,\delta_{r-1}}\log(n_1)\cdots\log(n_r)\]
for some $k$, then we can choose $q$ so large and $s$ in such a way that $I(q;s)=\{k\}$. This gives that 
\[\|F\|_{L^1([0,1])}\geq \frac{C_{MPS}}{32\pi(2+\log(1+2/\delta_1))}\|f_k\|_{L^1([0,1])},\]
yielding the theorem immediately. Thus there is no loss of generality in assuming
\begin{equation}\label{normBound}
    \|f_k\|_{L^1([0,1])}\leq C_{\delta_2,\ldots,\delta_{r-1}}\log(n_1)\cdots\log(n_r)
\end{equation}
for each $k\in I$.

Now we choose $q$ and $s$ as in Lemma \ref{goodModulus} applied to $I$ to get \[I(q;s)=\{k_1<\ldots<k_{J}\}\] satisfying
\[\frac{n_1^{1/3}}8\leq J\leq q^{1/2}.\] Then
\begin{equation}\label{estimate}
    \int_0^1\left|\sum_{a\in A}e(at)\right|dt\geq\frac{1}{2^5\pi(2+\log(1+2/\delta))}\sum_{j=1}^{J}\|f_{k_j}\|_{L^1([0,1])}\lr{\frac{C_{MPS}}{2j}-\frac{2\pi d_1}{qd_2}}.
\end{equation}
By induction,
\[\sum_{j=1}^{J}\frac{\|f_{k_j}\|_{L^1([0,1])}}{j}\geq C_{\delta_2,\ldots,\delta_{r-1}}\log(J)\log(n_2)\cdots\log(n_r)\geq \frac{1}{4}C_{\delta_2,\ldots,\delta_{r-1}}\log(n_1)\cdots\log(n_r).\]

By (\ref{normBound}), the error term in (\ref{estimate}) is at most
\[\frac{2\pi Jd_1}{qd_2}C_{\delta_2,\ldots,\delta_{r-1}}\log(n_1)\cdots\log(n_r)\leq \pi C_{\delta_2,\ldots,\delta_{r-1}}\frac{\log(n_1)\cdots\log(n_r)}{q^{1/2}}.\]
As guaranteed by Lemma \ref{goodModulus} and the hypotheses of the theorem, \[q^{1/2}\geq n_1^{1/3}/8\geq 16\pi C_{MPS}\log(n_1)\cdots\log(n_r).\]
So, we have shown
\begin{align*}
    \int_0^1\left|\sum_{a\in A}e(at)\right|dt&\geq\frac{1}{2^5\pi(2+\log(1+2/\delta))}\frac{C_{\delta_2,\ldots,\delta_{r-1}}}{2^4}\log(n_1)\cdots\log(n_r)\\
    &= C_{\delta_1,\ldots,\delta_{r-1}}\log(n_1)\cdots\log(n_r).
\end{align*}
\end{proof}

\end{document}